\documentclass{article}%
\usepackage{amsmath}
\usepackage{amsfonts}
\usepackage{amssymb}
\usepackage{graphicx}%
\newtheorem{theorem}{Theorem} [section]

\newtheorem{corollary}[theorem]{Corollary}

\newtheorem{lemma}[theorem]{Lemma}

\newenvironment{proof}[1][Proof]{\noindent\textbf{#1.} }{\ \rule{0.5em}{0.5em}}
\begin{document}

\author{Vadim E. Levit\\Department of Computer Science and Mathematics\\Ariel University Center of Samaria, ISRAEL\\levitv@ariel.ac.il
\and Eugen Mandrescu\\Department of Computer Science\\Holon Institute of Technology, ISRAEL\\eugen\_m@hit.ac.il}
\title{On Duality between Local Maximum Stable Sets of a Graph and its Line-Graph}
\date{}
\maketitle

\begin{abstract}
$G$ is a \textit{K\"{o}nig-Egerv\'{a}ry graph }provided $\alpha(G)+\mu
(G)=\left\vert V(G)\right\vert $, where $\mu(G)$ is the size of a maximum
matching and $\alpha(G)$ is the cardinality of a maximum stable set,
\cite{dem}, \cite{ster}.

$S$ is a \textit{local maximum stable set} of $G$, and we write $S\in\Psi(G)$,
if $S$ is a maximum stable set of the subgraph induced by $S\cup N(S)$, where
$N(S)$ is the neighborhood of $S$, \cite{LevMan2}. Nemhauser and Trotter Jr.
proved that any $S\in\Psi(G)$ is a subset of a maximum stable set of $G$,
\cite{NemhTro}.

In this paper we demonstrate that if $S\in\Psi(G)$, the subgraph $H$ induced
by $S\cup N(S)$ is a K\"{o}nig-Egerv\'{a}ry\ graph, and $M$ is a maximum
matching in $H$, then $M$ is a local maximum stable set in the line graph of
$G$.

\textbf{Keywords: } Line graph, K\"{o}nig-Egerv\'{a}ry\ graph, maximum
matching, local maximum stable set.

\end{abstract}

\section{Introduction}

Throughout this paper $G=(V,E)$ is a simple (i.e., a finite, undirected,
loopless and without multiple edges) graph with vertex set $V=V(G)$ and edge
set $E=E(G)$. If $X\subset V$, then $G[X]$ is the subgraph of $G$ spanned by
$X$. By $G-W$ we mean the subgraph $G[V-W]$, if $W\subset V(G)$. We also
denote by $G-F$ the partial subgraph of $G$ obtained by deleting the edges of
$F$, for $F\subset E(G)$, and we write shortly $G-e$, whenever $F$ $=\{e\}$.
If $A,B\subset V$ are disjoint and non-empty, then by $(A,B)$ we mean the set
$\{ab:ab\in E,a\in A,b\in B\}$. 

The \textit{neighborhood} of a vertex $v\in V$ is the set $N(v)=\{w:w\in V$
\ \textit{and} $vw\in E\}$. If $\left\vert N(v)\right\vert =1$, then $v$ is a
\textit{pendant vertex}. We denote the \textit{neighborhood} of $A\subset V$
by $N_{G}(A)=\{v\in V-A:N(v)\cap A\neq\emptyset\}$ and its \textit{closed
neighborhood} by $N_{G}[A]=A\cup N(A)$, or shortly, $N(A)$ and $N[A]$, if no
ambiguity. 

$K_{n},C_{n}$ denote respectively, the complete graph on $n\geq1$ vertices,
and the chordless cycle on $n\geq3$ vertices. A graph having no $K_{3}$ as a
subgraph is a \textit{triangle-free graph}.

A \textit{stable} set in $G$ is a set of pairwise non-adjacent vertices. A
stable set of maximum size will be referred to as a \textit{maximum stable
set} of $G$, and the \textit{stability number }of $G$, denoted by $\alpha(G)$,
is the cardinality of a maximum stable set in $G$. In the sequel, by
$\Omega(G)$ we denote the set of all maximum stable sets of the graph $G$. 

A set $A\subseteq V(G)$ is a \textit{local maximum stable set} of $G$ if $A$
is a maximum stable set in the subgraph spanned by $N[A]$, i.e., $A\in
\Omega(G[N[A]])$, \cite{LevMan2}. Let $\Psi(G)$ stand for the set of all local
maximum stable sets of $G$.

Clearly, every set $S$ consisting of only pendant vertices belongs to
$\Psi(G)$. Nevertheless, it is not a must for a local maximum stable set to
contain pendant vertices. For instance, $\{e,g\}\in\Psi(G)$, where $G$ is the
graph from Figure \ref{fig101}.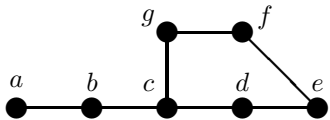
\begin{figure}[h]
\setlength{\unitlength}{1.0cm} \begin{picture}(5,1.5)\thicklines
\multiput(5,0)(1,0){5}{\circle*{0.29}}
\multiput(7,1)(1,0){2}{\circle*{0.29}}
\put(5,0){\line(1,0){4}}
\put(7,1){\line(1,0){1}}
\put(7,0){\line(0,1){1}}
\put(8,1){\line(1,-1){1}}
\put(5,0.35){\makebox(0,0){$a$}}
\put(6,0.35){\makebox(0,0){$b$}}
\put(6.75,0.3){\makebox(0,0){$c$}}
\put(8,0.35){\makebox(0,0){$d$}}
\put(6.75,1.2){\makebox(0,0){$g$}}
\put(8.3,1.2){\makebox(0,0){$f$}}
\put(9,0.3){\makebox(0,0){$e$}}
\end{picture}\caption{A graph having {various local maximum stable sets}.}%
\label{fig101}%
\end{figure}

The following theorem concerning maximum stable sets in general graphs, due to
Nemhauser and Trotter Jr. \cite{NemhTro}, shows that some stable sets can be
enlarged to maximum stable sets.

\begin{theorem}
\cite{NemhTro}\label{th1} Every local maximum stable set of a graph is a
subset of a maximum stable set.
\end{theorem}

Let us notice that the converse of Theorem \ref{th1} is trivially true,
because $\Omega(G)\subseteq\Psi(G)$. The graph $W$ from Figure \ref{fig101}
has the property that any $S\in\Omega(W)$ contains some local maximum stable
set, but these local maximum stable sets are of different cardinalities:
$\{a,d,f\}\in\Omega(W)$ and $\{a\},\{d,f\}\in\Psi(W)$, while for
$\{b,e,g\}\in\Omega(W)$ only $\{e,g\}\in\Psi(W)$.

However, there exists a graph $G$ satisfying $\Psi(G)=\Omega(G)$, e.g.,
$G=C_{n}$, for $n\geq4$.

A \textit{matching} in a graph $G=(V,E)$ is a set of edges $M\subseteq E$ such
that no two edges of $M$ share a common vertex. A \textit{maximum matching} is
a matching of maximum size, denoted by $\mu(G)$. A matching is \textit{perfect
}if it saturates all the vertices of the graph. A matching $M=\{a_{i}%
b_{i}:a_{i},b_{i}\in V(G),1\leq i\leq k\}$ of a graph $G$ is called \textit{a
uniquely restricted\emph{\ }matching} if $M$ is the unique perfect matching of
$G[\{a_{i},b_{i}:1\leq i\leq k\}]$, \cite{GolHiLew}. Recently, a
generalization of this concept, namely, a \textit{subgraph restricted
matching} has been studied in \cite{GoHeHeLa}.

Kroghdal found that a matching $M$ of a bipartite graph is uniquely restricted
if and only if $M$ is alternating cycle-free (see \cite{Krogdahl}). This
statement was observed for general graphs by Golumbic \textit{et al}. in
\cite{GolHiLew}.

In \cite{LevMan2}, \cite{LevMan45}, \cite{LevMan07}, \cite{LevMan08a},
\cite{LevMan08b} we showed that, under certain conditions involving uniquely
restricted matchings, $\Psi(G)$ forms a greedoid on $V(G)$. The classes of
graphs, where greedoids were found include trees, bipartite graphs,
triangle-free graphs, and well-covered graphs.

Recall that $G$ is a \textit{K\"{o}nig-Egerv\'{a}ry graph }provided
$\alpha(G)+\mu(G)=\left\vert V(G)\right\vert $ (\cite{dem}, \cite{ster}). As a
well-known example, any bipartite graph is a K\"{o}nig-Egerv\'{a}ry graph
(\cite{Eger}, \cite{Koen}). Properties of K\"{o}nig-Egerv\'{a}ry graphs were
discussed in a number of papers, e.g., \cite{bourpull}, \cite{kora},
\cite{KoNgPeis}, \cite{LevMan0}, \cite{LevMan1}, \cite{LevMan3}, \cite{lovpl},
\cite{pasdema}. Let us notice that if $S$ is a stable set and $M$ is a
matching in a graph $G$ such that $\left\vert S\right\vert +\left\vert
M\right\vert =\left\vert V(G)\right\vert $, it follows that $S\in\Omega(G),M$
is a maximum matching, and $G$ is a K\"{o}nig-Egerv\'{a}ry graph, because
$\left\vert S\right\vert +\left\vert M\right\vert \leq\alpha(G)+\mu
(G)\leq\left\vert V(G)\right\vert $ is true for any graph.

The line graph of a graph $G=(V,E)$ is the graph $L(G)=(E,U)$, where
$e_{i}e_{j}\in U$ if $e_{i},e_{j}$ have a common endpoint in $G$.

In this paper we give a sufficient condition in terms of subgraphs of $G$ that
ensure that its line graph $L(G)$ has proper local maximum stable sets. In
other words, we demonstrate that if: $S\in\Psi(G)$, the subgraph $H$ induced
by $S\cup N(S)$ is a K\"{o}nig-Egerv\'{a}ry\ graph, and $M$ is a maximum
matching in $H$, then $M$ is a local maximum stable set in the line graph of
$G$. It turns out that this is also a sufficient condition for a matching of
$G$ to be extendable to a maximum matching.

\section{Maximum matchings and local maximum stable sets}

In a K\"{o}nig-Egerv\'{a}ry graph, maximum matchings have a special property,
emphasized by the following statement.

\begin{lemma}
\cite{LevMan1}\label{match} Every maximum matching $M$ of a
K\"{o}nig-Egerv\'{a}ry graph $G$ is contained in each $(S,V(G)-S)$ and
$\left\vert M\right\vert =\left\vert V(G)-S\right\vert $, where $S\in
\Omega(G)$.
\end{lemma}

For example, $M=\{e_{1},e_{2},e_{3}\}$ is a maximum matching in the
K\"{o}nig-Egerv\'{a}ry graph $H$ (from Figure \ref{Fig11}), $S=\{a,b,c,d\}\in
\Omega(H)$ and $M\subset(S,V(H)-S)$. On the other hand, $M_{1}=\{xz,yv\},M_{2}%
=\{yz,uv\}$ are maximum matchings in the non-K\"{o}nig-Egerv\'{a}ry graph $G$
(depicted in Figure \ref{Fig11}), $S=\{x,y\}\in\Omega(G)$ and $M_{1}%
\subset(S,V(G)-S)$, while $M_{2}\nsubseteq(S,V(G)-S)$.

\begin{figure}[h]
\setlength{\unitlength}{1.0cm} \begin{picture}(5,2.5)\thicklines
\multiput(4,0)(1,0){2}{\circle*{0.29}}
\multiput(4,1)(2,0){2}{\circle*{0.29}}
\put(5,2){\circle*{0.29}}
\put(4,0){\line(1,0){1}}
\put(4,1){\line(1,0){2}}
\put(4,1){\line(1,1){1}}
\put(4,1){\line(1,-1){1}}
\put(5,0){\line(0,1){2}}
\put(5,0){\line(1,1){1}}
\put(5,2){\line(1,-1){1}}
\put(3.7,0){\makebox(0,0){$x$}}
\put(3.7,1){\makebox(0,0){$y$}}
\put(5.3,0){\makebox(0,0){$z$}}
\put(5.3,2){\makebox(0,0){$v$}}
\put(6.3,1){\makebox(0,0){$u$}}
\put(2.9,1){\makebox(0,0){$G$}}
\multiput(9,0)(1,0){3}{\circle*{0.29}}
\multiput(9,1)(1,0){2}{\circle*{0.29}}
\multiput(9,2)(1,0){2}{\circle*{0.29}}
\put(9,0){\line(1,0){2}}
\put(10,0){\line(0,1){2}}
\put(9,0){\line(1,1){1}}
\put(9,1){\line(1,0){1}}
\put(9,1){\line(1,-1){1}}
\put(9,2){\line(1,0){1}}
\put(9,2){\line(1,-1){2}}
\put(9.5,2.2){\makebox(0,0){$e_{1}$}}
\put(9.5,1.2){\makebox(0,0){$e_{2}$}}
\put(10.5,0.2){\makebox(0,0){$e_{3}$}}
\put(8.7,0){\makebox(0,0){$a$}}
\put(8.7,1){\makebox(0,0){$b$}}
\put(8.7,2){\makebox(0,0){$c$}}
\put(11.3,0){\makebox(0,0){$d$}}
\put(7.9,1){\makebox(0,0){$H$}}
\end{picture}\caption{$\{x,y\}\in\Omega(G)$ and $\{a,b,c,d\}\in\Omega(H)$.}%
\label{Fig11}%
\end{figure}
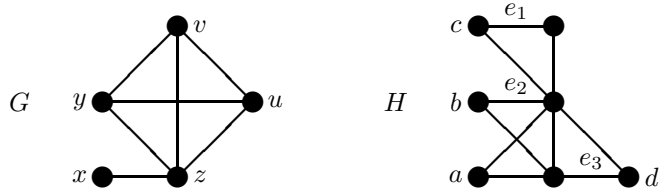

Clearly, (maximum) matchings in a graph $G$ correspond to (maximum,
respectively) stable sets in $L(G)$ and vice versa. However, not every
matching $M$ in $G$ gives birth to a local maximum stable set in $L(G)$, even
if $M$ can be enlarged to a maximum matching. 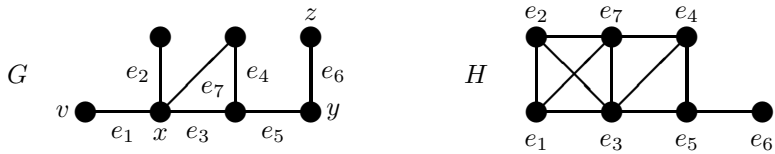
\begin{figure}[h]
\setlength{\unitlength}{1.0cm} \begin{picture}(5,1.9)\thicklines
\multiput(3,0.5)(1,0){4}{\circle*{0.29}}
\multiput(4,1.5)(1,0){3}{\circle*{0.29}}
\put(3,0.5){\line(1,0){3}}
\put(4,0.5){\line(0,1){1}}
\put(4,0.5){\line(1,1){1}}
\put(5,0.5){\line(0,1){1}}
\put(6,0.5){\line(0,1){1}}
\put(3.5,0.2){\makebox(0,0){$e_{1}$}}
\put(3.7,1){\makebox(0,0){$e_{2}$}}
\put(4.5,0.2){\makebox(0,0){$e_{3}$}}
\put(5.5,0.2){\makebox(0,0){$e_{5}$}}
\put(4.7,0.8){\makebox(0,0){$e_{7}$}}
\put(5.3,1){\makebox(0,0){$e_{4}$}}
\put(6.3,1){\makebox(0,0){$e_{6}$}}
\put(4,0.2){\makebox(0,0){$x$}}
\put(6.3,0.5){\makebox(0,0){$y$}}
\put(6,1.8){\makebox(0,0){$z$}}
\put(2.7,0.5){\makebox(0,0){$v$}}
\put(2.1,1){\makebox(0,0){$G$}}
\multiput(9,0.5)(1,0){4}{\circle*{0.29}}
\multiput(9,1.5)(1,0){3}{\circle*{0.29}}
\put(9,0.5){\line(1,0){3}}
\put(9,1.5){\line(1,0){2}}
\multiput(9,0.5)(1,0){3}{\line(0,1){1}}
\multiput(9,0.5)(1,0){2}{\line(1,1){1}}
\put(9,1.5){\line(1,-1){1}}
\put(9,0.1){\makebox(0,0){$e_{1}$}}
\put(9,1.8){\makebox(0,0){$e_{2}$}}
\put(10,0.1){\makebox(0,0){$e_{3}$}}
\put(11,0.1){\makebox(0,0){$e_{5}$}}
\put(10,1.8){\makebox(0,0){$e_{7}$}}
\put(11,1.8){\makebox(0,0){$e_{4}$}}
\put(12,0.1){\makebox(0,0){$e_{6}$}}
\put(8.2,1){\makebox(0,0){$H$}}
\end{picture}\caption{The graph $G$ and its line-graph $H=L(G)$.}%
\label{fig1}%
\end{figure}

For instance, $M_{1}=\{e_{1},e_{6}\},M_{2}=\{e_{3},e_{6}\}$ are both matchings
in the graph $G$ from Figure \ref{fig1}, but only $M_{1}$ is a local maximum
stable set in $L(G)$. Remark that $S_{1}=\{v,z\}\in$ $\Psi(G)$, $S_{2}%
=\{x,y\}\notin\Psi(G)$ and each $M_{i}$ is a maximum matching in $G[N[S_{i}%
]]$, for $i\in\{1,2\}$.

\begin{theorem}
\label{th2}If $S\in\Psi(G),H=G[N[S]]$ is a K\"{o}nig-Egerv\'{a}ry graph, and
$M$ is a maximum matching in $H$, then $M$ is a local maximum stable set in
$L(G)$.
\end{theorem}

\begin{proof}
Let $M=\{e_{i}=v_{i}w_{i}:1\leq i\leq\mu(H)\}$. According to Lemma
\ref{match}, it follows that
\[
M\subseteq(S,V(H)-\nolinebreak S)\text{ and }\left\vert M\right\vert
=\left\vert V(H)-\nolinebreak S\right\vert ,
\]
because $H$ is a K\"{o}nig-Egerv\'{a}ry graph. Consequently, without loss of
generality, we may suppose that
\[
\{v_{i}:1\leq i\leq\mu(H)\}\subseteq S\text{, while }V(H)-\nolinebreak
S=\{w_{i}:1\leq i\leq\mu(H)\}.
\]
Since $N_{H}(v_{i})=N_{G}(v_{i})\subseteq N(S)=V(H)-\nolinebreak S$, we have
that
\[
N_{L(G)}[M]=E(H)\cup\{e=wt\in E:w\in V(H)-S,t\notin S\}.
\]
Hence, every $e\in N_{L(G)}[M]-V(L(H))$ is incident in $G$ to some $w_{i}$.

Assume that $M$ is not a maximum stable set in $L(G)$, i.e., there exists some
stable set $Q\subseteq N_{L(G)}[M]$, such that $\left\vert Q\right\vert
>\left\vert M\right\vert $. In other words, $Q$ is a matching using edges
from
\[
E(H)\cup\{e=wt\in E:w\in V(H)-S,t\notin S\},
\]
larger than $M$. Let $F=(M-Q)\cup(Q-M)$. Since $M$ and $Q$ are matchings,
every vertex appearing in $G[F]$ has at most one incident edge from each of
them, and the maximum degree of a vertex in $G[F]$ is $2$. Hence, $G[F]$
consists of only disjoints chordless paths and cycles. Moreover, every path
and every cycle in $G[F]$ alternates between edges of $Q$ and edges of $M$.
Since $\left\vert Q\right\vert >\left\vert M\right\vert $, it follows that
$G[F]$ has a component with more edges of $Q$ than of $M$. Such a component
can only be a path, say $P_{x,y}$, that starts and ends with edges from $Q$
(more precisely, from $Q-M$) and and $x,y$ are not saturated by edges
belonging to $M$. Hence, $P_{x,y}$ must have an odd number of edges.

\emph{Case 1.} $P_{x,y}$ contains only one edge, namely $xy$. This is not
possible, since at least one of the vertices $x,y$ belongs to $V(H)-S$ and is
saturated by $M$.

\emph{Case 2.} $P_{x,y}$ contains at least three edges.

Let $xa,by\in Q$ be the first and the last edges on $P_{x,y}$. Clearly,
$E(P_{x,y})\nsubseteq E(H)$, because, otherwise
\[
(M-E(P_{x,y}))\cup(E(P_{x,y})-M)
\]
is a matching in $H$, larger than $M$, in contradiction with the maximality of
$M$. Hence, $P_{x,y}$ contains edges from $M$, that alternates with edges from
$(E(H)-M)\cup W$, where
\[
W=\{wt\in E(G):w\in V(H)-\nolinebreak S,t\in U\},
\]
with
\[
U=(S-\{v_{i}:1\leq i\leq\mu(H)\})\cup(V(G)-V(H)\}\neq\varnothing.
\]

Therefore, each second vertex on $P_{x,y}$ must belong to $V(H)-S$.
Consequently, we infer that also $y\in V(H)-S$, and hence, it is saturated by
$M$, a contradiction.
\end{proof}

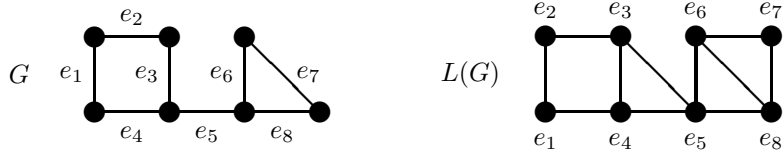
\begin{figure}[h]
\setlength{\unitlength}{1.0cm} \begin{picture}(5,1.8)\thicklines
\multiput(3,0.5)(1,0){4}{\circle*{0.29}}
\multiput(3,1.5)(1,0){3}{\circle*{0.29}}
\put(3,0.5){\line(1,0){3}}
\put(4,0.5){\line(0,1){1}}
\put(3,1.5){\line(1,0){1}}
\put(5,0.5){\line(0,1){1}}
\put(3,0.5){\line(0,1){1}}
\put(5,1.5){\line(1,-1){1}}
\put(2.7,1){\makebox(0,0){$e_{1}$}}
\put(3.5,0.2){\makebox(0,0){$e_{4}$}}
\put(3.7,1){\makebox(0,0){$e_{3}$}}
\put(3.5,1.75){\makebox(0,0){$e_{2}$}}
\put(4.5,0.2){\makebox(0,0){$e_{5}$}}
\put(4.7,1){\makebox(0,0){$e_{6}$}}
\put(5.85,1){\makebox(0,0){$e_{7}$}}
\put(5.5,0.2){\makebox(0,0){$e_{8}$}}
\put(2,1){\makebox(0,0){$G$}}
\multiput(9,0.5)(1,0){4}{\circle*{0.29}}
\multiput(9,1.5)(1,0){4}{\circle*{0.29}}
\put(9,0.5){\line(1,0){3}}
\put(9,0.5){\line(0,1){1}}
\put(9,1.5){\line(1,0){1}}
\put(10,0.5){\line(0,1){1}}
\put(10,1.5){\line(1,-1){1}}
\put(11,0.5){\line(0,1){1}}
\put(11,1.5){\line(1,0){1}}
\put(11,1.5){\line(1,-1){1}}
\put(12,0.5){\line(0,1){1}}
\put(9,0.1){\makebox(0,0){$e_{1}$}}
\put(10,0.1){\makebox(0,0){$e_{4}$}}
\put(10,1.85){\makebox(0,0){$e_{3}$}}
\put(9,1.85){\makebox(0,0){$e_{2}$}}
\put(11,0.1){\makebox(0,0){$e_{5}$}}
\put(11,1.85){\makebox(0,0){$e_{6}$}}
\put(12,1.85){\makebox(0,0){$e_{7}$}}
\put(12,0.1){\makebox(0,0){$e_{8}$}}
\put(8,1){\makebox(0,0){$L(G)$}}
\end{picture}\caption{$M=\{e_{5},e_{7}\}$ is a matching in $G$ and local
maximum stable set in $L(G)$.}%
\label{Fig1}%
\end{figure}

Notice that $M=\{e_{5},e_{7}\}\in\Psi(L(G))$, while there is no $S\in\Psi(G)$,
such that $M$ is a maximum matching in $G[N[S]]$, where $G$ is depicted in
Figure \ref{Fig1}. In other words, the converse of Theorem \ref{th2} is not true.

Clearly, every matching can be enlarged to a maximal matching, which is not
necessarily a maximum matching. For instance, the graph $G$ in Figure
\ref{fig30} does not contain any maximum matching including the matching
$M=\{e_{0},e_{1},e_{2}\}$. The following result shows that, under certain
conditions, a matching\emph{ }can be extended to a maximum
matching.\begin{figure}[h]
\setlength{\unitlength}{1.0cm} \begin{picture}(5,1.5)\thicklines
\multiput(5,0)(2,0){2}{\circle*{0.29}}
\multiput(5,0.5)(2,0){2}{\circle*{0.29}}
\multiput(6,1)(3,0){2}{\circle*{0.29}}
\multiput(8,0)(2,0){2}{\circle*{0.29}}
\multiput(8,0.5)(2,0){2}{\circle*{0.29}}
\multiput(5,0)(2,0){2}{\line(0,1){0.5}}
\multiput(8,0)(2,0){2}{\line(0,1){0.5}}
\multiput(5,0.5)(3,0){2}{\line(2,1){1}}
\multiput(6,1)(3,0){2}{\line(2,-1){1}}
\multiput(5,0)(3,0){2}{\line(1,0){2}}
\put(6,1){\line(1,0){3}}
\put(7.5,1.2){\makebox(0,0){$e_{0}$}}
\put(6,0.2){\makebox(0,0){$e_{1}$}}
\put(9,0.2){\makebox(0,0){$e_{2}$}}
\put(4.2,0.5){\makebox(0,0){$G$}}
\end{picture}\caption{$\{e_{0},e_{1},e_{2}\}$ is a maximal but not a maximum
matching.}%
\label{fig30}%
\end{figure}
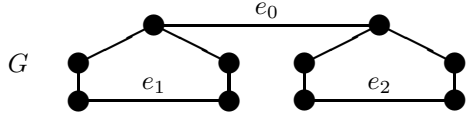

\begin{corollary}
\label{cor1}If $S\in\Psi(G),H=G[N[S]]$ is a K\"{o}nig-Egerv\'{a}ry graph, and
$M$ is a maximum matching in $H$, then there exists a maximum matching $M_{0}$
in $G$ such that $M\subseteq M_{0}$.
\end{corollary}

\begin{proof}
According to Theorem \ref{th2}, $M$ is a local maximum stable set in $L(G)$.
By Theorem \ref{th1}, there is some $M_{0}\in\Omega(L(G))$, such that $M$
$\subseteq M_{0}$. Hence, $M_{0}$ is a maximum matching in $G$ containing $M$.
\end{proof}

Let us notice that Corollary \ref{cor1} can not be generalized to any subgraph
of a non-bipartite K\"{o}nig-Egerv\'{a}ry graph. 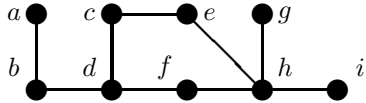
\begin{figure}[h]
\setlength{\unitlength}{1.0cm} \begin{picture}(5,1.5)\thicklines
\multiput(5,0)(1,0){5}{\circle*{0.29}}
\multiput(5,1)(1,0){4}{\circle*{0.29}}
\put(5,0){\line(1,0){4}}
\put(6,0){\line(0,1){1}}
\put(6,1){\line(1,0){1}}
\put(7,1){\line(1,-1){1}}
\put(5,0){\line(0,1){1}}
\put(8,0){\line(0,1){1}}
\put(4.7,1){\makebox(0,0){$a$}}
\put(4.7,0.3){\makebox(0,0){$b$}}
\put(5.7,1){\makebox(0,0){$c$}}
\put(5.7,0.3){\makebox(0,0){$d$}}
\put(7.3,1){\makebox(0,0){$e$}}
\put(6.7,0.3){\makebox(0,0){$f$}}
\put(8.3,1){\makebox(0,0){$g$}}
\put(8.3,0.3){\makebox(0,0){$h$}}
\put(9.3,0.3){\makebox(0,0){$i$}}
\end{picture}
\caption{$M=\{ab,cd,fh\}$ is a maximum matching in $N[\{a,c,f\}]$.}%
\label{fig2929}%
\end{figure}

For instance, the graph $G$ depicted in Figure \ref{fig2929} is a
K\"{o}nig-Egerv\'{a}ry graph, $S=\{a,c,f\}\in\Psi(G)$, and $M=\{ab,cd,fh\}$ is
a maximum matching in $G[N[S]]$, which is not a K\"{o}nig-Egerv\'{a}ry graph,
but there is no maximum matching in $G$ that includes $M$.

Since any subgraph of a bipartite graph is also bipartite, we obtain the
following result.

\begin{corollary}
\label{cor2}If $G$ is a bipartite graph, $S\in\Psi(G)$ and $M$ is a maximum
matching in $G[N[S]]$, then there exists a maximum matching $M_{0}$ in $G$
such that $M\subseteq M_{0}$.
\end{corollary}

\section{Conclusions}

We showed that there is some connection between $\Psi(G)$ and $\Psi(L(G))$.

Let us notice that there are graphs whose line graphs have no proper local
maximum stable sets (see the graphs in Figure \ref{fig29}).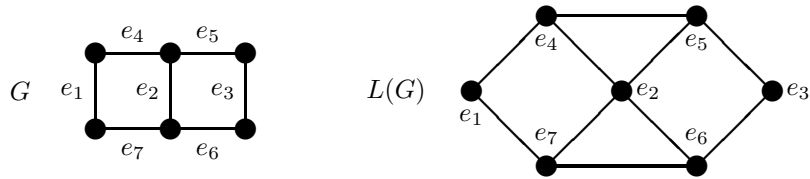
\begin{figure}[h]
\setlength{\unitlength}{1.0cm} \begin{picture}(5,2)\thicklines
\multiput(3,0.5)(1,0){3}{\circle*{0.29}}
\multiput(3,1.5)(1,0){3}{\circle*{0.29}}
\put(3,0.5){\line(1,0){2}}
\put(4,0.5){\line(0,1){1}}
\put(3,1.5){\line(1,0){2}}
\put(5,0.5){\line(0,1){1}}
\put(3,0.5){\line(0,1){1}}
\put(4,0.5){\line(0,1){1}}
\put(2.7,1){\makebox(0,0){$e_{1}$}}
\put(3.5,0.2){\makebox(0,0){$e_{7}$}}
\put(3.7,1){\makebox(0,0){$e_{2}$}}
\put(4.5,1.75){\makebox(0,0){$e_{5}$}}
\put(4.5,0.2){\makebox(0,0){$e_{6}$}}
\put(4.7,1){\makebox(0,0){$e_{3}$}}
\put(3.5,1.75){\makebox(0,0){$e_{4}$}}
\put(2,1){\makebox(0,0){$G$}}
\multiput(9,0)(2,0){2}{\circle*{0.29}}
\multiput(8,1)(2,0){3}{\circle*{0.29}}
\multiput(9,2)(2,0){2}{\circle*{0.29}}
\put(9,0){\line(1,0){2}}
\put(9,2){\line(1,0){2}}
\put(9,0){\line(1,1){2}}
\put(9,2){\line(1,-1){2}}
\put(8,1){\line(1,1){1}}
\put(8,1){\line(1,-1){1}}
\put(11,0){\line(1,1){1}}
\put(11,2){\line(1,-1){1}}
\put(8,0.65){\makebox(0,0){$e_{1}$}}
\put(10.35,1){\makebox(0,0){$e_{2}$}}
\put(12.35,1){\makebox(0,0){$e_{3}$}}
\put(9,1.65){\makebox(0,0){$e_{4}$}}
\put(11,1.65){\makebox(0,0){$e_{5}$}}
\put(11,0.4){\makebox(0,0){$e_{6}$}}
\put(9,0.4){\makebox(0,0){$e_{7}$}}
\put(7,1){\makebox(0,0){$L(G)$}}
\end{picture}\caption{Both $G$ and its line graph $L(G)$ have no local maximum
stable sets.}%
\label{fig29}%
\end{figure}

Moreover, there are graphs whose iterated line graphs have no proper local
maximum stable set, e.g., each $C_{n}$, for $n\geq3$, since $C_{n}$ and
$L(C_{n})$ are isomorphic.

An interesting open question reads as follows. Is it true that for a connected
graph $G$\ the fact that $L(G)$ has no proper local maximum stable sets
implies that $G$ itself does not contain proper local maximum stable sets?

\end{document}